\RequirePackage[l2tabu, orthodox]{nag}
\RequirePackage{snapshot}

\documentclass[11pt,onecolumn]{article}

\makeatletter
\if@twocolumn
  \usepackage[dvips,letterpaper,top=0.5in, bottom=0.5in, left=0.75in, right=0.5in,includefoot,heightrounded]{geometry}
\else
  \usepackage[dvips,letterpaper,margin=1in,includefoot,heightrounded]{geometry}
\fi

\usepackage{srcltx}

\usepackage{amsmath}
\usepackage{amssymb,amsfonts}

\usepackage{abstract}

\usepackage{epsfig}
\usepackage[usenames,dvipsnames]{color}
\usepackage[usenames,dvipsnames]{xcolor}
\usepackage{subfigure}

\usepackage{booktabs}

\usepackage{setspace}
\usepackage{flushend}
\usepackage{multicol}

\usepackage{cite}
\usepackage{url}
\usepackage[normalem]{ulem}

\usepackage{enumerate}

\usepackage{hyperref}

\usepackage[sc,tiny]{titlesec}
       \usepackage{mathptmx}

\newtheorem{proposition}{Proposition}
\newtheorem{definition}{Definition}

\newtheorem{lemma}{Lemma}
\newtheorem{corollary}{Corollary}

\newtheorem{example}{Example}

\newcommand{\dens}{\operatorname{dens}}
\newcommand{\asymdens}{\operatorname{d}}

\newcommand{\dirichlet}{\operatorname{\partial}}
\newcommand{\supv}{\operatorname{sup_v}}
\newcommand{\suph}{\operatorname{sup_h}}

\newcommand{\positiveset}{{\mathbb{P}}}

\newcommand{\lcm}{\operatorname{lcm}}

\newcommand{\UPPER}[2]{{U}_{#1,#2}}
\newcommand{\LOWER}[2]{{L}_{#1,#2}}

\newcommand{\ROW}[2]{{#1_{\ast,#2}}}
\newcommand{\COL}[2]{{#1_{#2,\ast}}}



\title{%
An Extension of the Dirichlet Density for Sets of Gaussian Integers
}

\author{%
L.~C.~R\^ego%
\thanks{%
L.~C.~R\^ego is with 
the
Departamento de Estat\'{\i}stica, 
Universidade Federal de Pernambuco.
Email: 
\protect\url{leandro@de.ufpe.br}
}
\and
R.~J.~Cintra%
\thanks{%
R.~J.~Cintra is with 
the Signal Processing Group,
Departamento de Estat\'{\i}stica, 
Universidade Federal de Pernambuco.
Part of
second
author's work was done during his sabbatical at the
University of Calgary, Calgary, Canada.
Email: 
\protect\url{rjdsc@de.ufpe.br}
}
}

\date{}

\newcommand{\myabstract}{%
Several measures for the density of sets of integers have been proposed,
such as the asymptotic density, the Schnirelmann density, and the Dirichlet density.
There has been some work in the literature on extending some of these concepts of density to higher dimensional sets of integers.
In this work, we propose an extension of the Dirichlet density for sets of Gaussian integers and
investigate some of its properties.
}

\newcommand{\mykeywords}{%
Gaussian integers, Dirichlet density
}

\begin{document}

\makeatletter
\if@twocolumn

\twocolumn[%
  \maketitle
  \begin{onecolabstract}
    \myabstract
  \end{onecolabstract}
  \begin{center}
    \small
    \textbf{Keywords}
    \\\medskip
    \mykeywords
  \end{center}
  \bigskip
]
\saythanks

\else

  \maketitle
  \begin{abstract}
    \myabstract
  \end{abstract}
  \begin{center}
    \small
    \textbf{Keywords}
    \\\medskip
    \mykeywords
  \end{center}
  \bigskip
  \onehalfspacing
\fi

\section{Introduction}

Several measures for the density of sets of integers have been discussed in the
literature~\cite{sun2008density,bell2006,duncan1968,duncan1970,erdos1948,davenport1951,ahlswede1997}.
Presumably the most employed
of such measures is
the asymptotic density, also referred to as natural density~\cite{niven1951asymptotic,sun2008density}.
For a given set of integers $A$,
its asymptotic density
is expressed by
\begin{align*}
\asymdens(A) =
\lim_{n\to\infty}
\frac{\|A \cap \{1, 2, 3, \ldots, n\}\|}{n},
\end{align*}
provided that such a limit does exist.
The symbol $\|\cdot\|$ returns the cardinality of its argument.

In~\cite{bell2006}, Bell and Burris
bring an ample exposition on the Dirichlet density,
which is defined as follows.

\begin{definition}
The Dirichlet density of a subset $A$ of the positive integers
is given by
\begin{align*}
\dirichlet(A)
&\triangleq
\lim_{s\downarrow1}
\frac{\sum_{n\in A}\frac{1}{n^s}}
{\zeta(s)},
\end{align*}
if the limit does exist, for real $s>1$.
The quantity $\zeta(\cdot)$ denotes the Riemann zeta function~\cite{grad1965}.
\end{definition}
If the asymptotic density is well defined,
then the Dirichlet density does also exist and
assumes the same value~\cite[p.~10]{bateman2004}.
Since the converse is not always true,
the Dirichlet density is a more encompassing tool
when compared to the asymptotic density~\cite[p.~11]{bateman2004}.
Dirichlet density also admits lower and upper versions,
which have been explored along with other densities
to characterize primitive sets~\cite{ahlswede1996,ahlswede1999,ahlswede2004}.

Gaussian integers are simply complex numbers of the form $m + in$, where $m$ and $n$ are integers.
Despite the considerable amount
of development addressing densities for sets of positive integers~\cite{fuchs1990},
densities for sets of Gaussian integers appear to be an overlooked topic.
However, a seminal paper by Cheo~\cite{cheo1951}
investigated the question,
suggesting an extension of the Schnirelmann density~\cite{duncan1968,duncan1970}.
Such extended definition applies
to subsets of the nonzero Gaussian integers
inclusively confined
in the first quadrant of the complex plane.

Generalizations of Schnirelmann density for the $n$-dimensional case were proposed in~\cite{kvarda1963,kvarda1966}.
Additionally, a modified Schnirelmann density was introduced in~\cite{stalley1955}
and was generalized in~\cite{freedman1970} years later. In a comparable venue,
Freedman~\cite{freedman1969,freedman1973} advanced the concept of asymptotic density to higher dimensions.

In this context,
the aim of the present work is to advance a method for evaluating
the density of sets of Gaussian integers.
To address this problem,
a density based on Dirichlet generating functions is proposed.

For ease of notation,
henceforth we identify
a Gaussian integer $m+i n$ with the pair of integers $(m,n)$.
All considered Gaussian integers and their sets are defined in $\positiveset^2$,
where $\positiveset$ is the set of strictly positive integers.

\section{Definition and General Properties}

The Gaussian integers can be realized as points over a square lattice in the complex plane.
The square lattice is composed by an infinite array of Gaussian integers,
set up in rows and columns. In addition,
each lattice row or column can furnish sets of integers according to
the following constructions:
$\ROW{A}{n} = \{ m \in \positiveset : (m,n) \in A \}$
and
$\COL{A}{m} = \{ n \in \positiveset : (m,n) \in A \}$.

Our goal is to investigate
the properties of the following density for Gaussian integers,
which we show to be a generalization of the Dirichlet density for sets of integers.

\begin{definition}
Let $A$ be a set of Gaussian integers.
Admit $I_\ROW{A}{n}(\cdot)$ and $I_\COL{A}{m}(\cdot)$ to be the indicator functions of the sets
$\ROW{A}{n}$ and $\COL{A}{m}$, for $m,n\in\positiveset$, respectively.
The proposed density for $A$ is given by
\begin{align*}
\dens(A)
\triangleq
\lim_{s\downarrow1}
\frac{1}{\zeta^2(s)}
\sum_{m=1}^\infty
\sum_{n=1}^\infty
\frac{I_\ROW{A}{n}(m) I_\COL{A}{m}(n)}{(mn)^s},
\end{align*}
provided that the limit exists.
\end{definition}

\noindent
From now on,
we only consider sets whose referred densities are well-defined, i.e.,
the implied limits exist.
Thus, we restrain ourselves of indicating in every instance that the results are valid only when
the discussed limits exist.
In account of the proposed definition,
a series of consequences is listed below.

\begin{proposition}
\label{proposition.facts}
Let $A$ and $B$ be two sets of Gaussian integers.
The following assertions hold true:
\begin{enumerate}[(i)]
\item
$\dens(A) \geq 0$.
\item
$\dens(\positiveset^2) = 1$.
\item
if $A\cap B = \varnothing$, then $\dens(A\cup B)=\dens(A)+\dens(B)$.
\item
$\dens(\varnothing) = 0$.
\item 
$\dens(B-A) = \dens(B) - \dens(A\cap B)$,
where $B-A$ is the relative complement of $A$ in $B$.
\item 
$\dens(A^c) = 1 - \dens(A)$,
where $A^c$ is the complement of $A$.
\item if $A\subset B$, then $\dens(A) \leq \dens(B)$.
\item 
$
\dens(A \cup B) = \dens(A) + \dens(B) - \dens(A \cap B)
$.
\end{enumerate}
\end{proposition}

\begin{proof}
Follows directly from the definition.
\end{proof}

The
first
three properties stated in the previous proposition are
the same conditions that form an axiomatic definition of a
probability measure,
except for the $\sigma$-additivity axiom.

\begin{proposition}[Cartesian Product]
\label{proposition.cartesian.product}
Let $A$ and $B$ be two subsets of $\positiveset$.
Then the density of the Cartesian product
$A \times B$ satisfies
\begin{align*}
\dens(A\times B)=\dirichlet(A)\dirichlet(B).
\end{align*}
\end{proposition}

\begin{proof}
We have that
\begin{align*}
\dens(A\times B)
&=
\lim_{s\downarrow1}
\frac{\sum_{(m,n)\in A\times B}\frac{1}{(mn)^s}}
{\zeta^2(s)}
\\
&=
\lim_{s\downarrow1}
\frac{\sum_{m\in A}\frac{1}{m^s}\sum_{n\in B}\frac{1}{n^s}}
{\zeta^2(s)}
\\
&=
\dirichlet(A)\dirichlet(B).
\end{align*}
\end{proof}

\begin{corollary}[Dirichlet Density]
Let $A$ be a set of positive integers.
Then
$\dens(A\times\positiveset) = \dirichlet(A)$.
\end{corollary}

\begin{proof}
This result is a direct consequence of the fact that $\dirichlet(\positiveset)=1$~\cite{fuchs1990}.
\end{proof}

Given any set $A$ of Gaussian integers,
let the horizontal and vertical axis sections be
denoted by
$\suph(A) = \bigcup_{n=1}^\infty \ROW{A}{n}$
and
$\supv(A) = \bigcup_{m=1}^\infty \COL{A}{m}$,
respectively.

\begin{proposition}
Let $A$ be a set of Gaussian integers.
If $\dirichlet(\suph(A))=0$ or $\dirichlet(\supv(A))=0$,
then
$\dens(A)=0$.
\end{proposition}

\begin{proof}
For instance,
assume that $\dirichlet(\suph(A))=0$.
Note that $A\subset \suph(A) \times \positiveset$.
Due to the monotonicity property,
it follows that
$\dens(A) \leq \dens(\suph(A)\times\positiveset)$.
Moreover,
the property of the density of Cartesian products allows us to write
$\dens(A) \leq \dirichlet(\suph(A))\dirichlet(\positiveset)$.
Applying the hypothesis, the result follows.
The proof would be analogous in the case that $\dirichlet(\supv(A))=0$.
\end{proof}

\begin{corollary}[Finite Axis Section]
If a set $A$ of Gaussian integers
has a finite axis section,
then $\dens(A)=0$.
\end{corollary}

\begin{proof}
It is enough to observe that any finite set of integers has null Dirichlet density~\cite{fuchs1990}.
\end{proof}
 \noindent
 As a consequence,
 a finite set of Gaussian integers has null density, since both of its axis sections are finite.
 In particular, the density of a singleton is zero.
 On the other hand,
 nonzero density subsets must have infinite axis sections.

 Let
 $\UPPER{m_0}{n_0} = \{(m,n)\in \positiveset^2:m\geq m_0\text{ and }n\geq n_0\}$
 and
 $\LOWER{m_0}{n_0} = \{(m,n)\in \positiveset^2:m < m_0\text{ and } n < n_0\}$.
 Next proposition states that,
 for density evaluation,
 the only relevant set elements
 are those
 located in the region defined by $\UPPER{m_0}{n_0}$ for any choice of $m_0$ and $n_0$.
 This means that the ``weight'' of the set is located on its ``tail''.
 Nevertheless,
 we need the result of the following lemma before.

 \begin{lemma}
 The set $\UPPER{m_0}{n_0}$ has unit density.
 \end{lemma}

 \begin{proof}
 Let $\UPPER{m_0}{n_0}^c$ be the complement of $\UPPER{m_0}{n_0}$.
 Therefore,
 the set $\positiveset^2$ can be partitioned into
 $\positiveset^2 = \UPPER{m_0}{n_0} \cup \UPPER{m_0}{n_0}^c$.
 Then,
 it follows that
 $\dens(\UPPER{m_0}{n_0}) = 1 - \dens(\UPPER{m_0}{n_0}^c)$.
 Notice also that $\UPPER{m_0}{n_0}^c = \LOWER{m_0}{\infty} \cup \LOWER{\infty}{n_0}$.
The union property
allows us to state that
 $\dens(\UPPER{m_0}{n_0}^c) = \dens(\LOWER{m_0}{\infty}) + \dens(\LOWER{\infty}{n_0}) - \dens(\LOWER{m_0}{n_0})$.
 Since
 $\LOWER{m_0}{\infty}$,
 $\LOWER{\infty}{n_0}$,
 and
 $\LOWER{m_0}{n_0}$
 have each at least one finite axis section,
 it follows that $\dens(\UPPER{m_0}{n_0}^c) = 0$.

 \end{proof}

 \begin{proposition}[Heavy Tail]
 \label{propositon.heavy}
 Let $A$ be a set of Gaussian integers.
 Then,
 for any two given nonnegative integers $m_0$ and $n_0$,
 we have
 \begin{align*}
 \dens(A) = \dens(A \cap \UPPER{m_0}{n_0}).
 \end{align*}
 \end{proposition}

 \begin{proof}
 Observe that $A = A \cap (\UPPER{m_0}{n_0} \cup \UPPER{m_0}{n_0}^c)=(A\cap\UPPER{m_0}{n_0}) \cup (A\cap\UPPER{m_0}{n_0}^c)$.
 Since we have a partition of $A$,
 it follows that
 $\dens(A) = \dens(A\cap\UPPER{m_0}{n_0}) + \dens(A\cap \UPPER{m_0}{n_0}^c)$.
 But, $A\cap \UPPER{m_0}{n_0}^c \subset \UPPER{m_0}{n_0}^c$,
 then
 $\dens(A\cap \UPPER{m_0}{n_0}^c) \leq \dens(\UPPER{m_0}{n_0}^c) = 0$.
 \end{proof}

 \begin{proposition}[Axis Independence]
 If there is a pair $(m_0,n_0)$ such that,
 for every $m\geq m_0$ and $n\geq n_0$,
 the functions $I_\COL{A}{m}(n)$ and $I_\ROW{A}{n}(m)$
 are independent of $m$ and $n$, respectively,
 then
 \begin{align*}
 \dens(A)
 =
 \dirichlet(\COL{A}{m_0})
 \dirichlet(\ROW{A}{n_0}).
 \end{align*}
 \end{proposition}

 \begin{proof}
 Because of the assumed independence,
 we can write
 $I_\COL{A}{m}(n) = I_\COL{A}{m_0}(n)$
 and
 $I_\ROW{A}{n}(m) = I_\ROW{A}{n_0}(m)$,
 for
 $m\geq m_0$ and
 $n\geq n_0$, respectively.
 Thus,
 for $m\geq m_0$ and
 $n\geq n_0$,
 the set $A$ is
 indistinguishable of
 $\COL{A}{m_0} \times \ROW{A}{n_0}$.
 But, the heavy tail property implies that
 \begin{align*}
 \dens(A)
 &=
 \dens(A \cap \UPPER{m_0}{n_0})
 \\
 &=
 \dens((\COL{A}{m_0} \times \ROW{A}{n_0}) \cap \UPPER{m_0}{n_0})
 \\
 &=
 \dens(\COL{A}{m_0} \times \ROW{A}{n_0})
 \\
 &=
 \dirichlet(\COL{A}{m_0})
 \dirichlet(\ROW{A}{n_0}).
 \end{align*}

 \end{proof}

Given a set $A$ of Gaussian integers and a Gaussian integer $(m_0,n_0)$,
let $A\oplus (m_0,n_0) \triangleq \{ (m + m_0, n+n_0) \ |\ (m,n)\in A\}$.
This process is called a translation of $A$ by $(m_0,n_0)$ units~\cite[p.~49]{rudin1966}.
Now our goal is to show that the proposed density is translation invariant, i.e.,
$\dens(A \oplus (m_0,n_0)) = \dens(A)$, $m_0\geq0$ and $n_0\geq 0$.
However, the proof that we will supply requires the following lemma.

\begin{lemma}[Unitary Translation]
Let $A$ be a set of Gaussian integers, such as $\dens(A)>0$.
Then
\begin{align*}
\dens(A \oplus (1,0))=\dens(A \oplus (0,1))= \dens(A).
\end{align*}
\end{lemma}

\begin{proof}
It suffices to show that $\dens(A \oplus (1,0)) = \dens(A)$,
being the other case analogous.
First, note that since
\begin{align*}
\sum_{(m,n)\in A}\frac{1}{(mn)^s}-\sum_{(m,n)\in A}\frac{1}{((m+1)n)^s}\geq 0,
\end{align*}
it follows that $\dens(A)-\dens(A\oplus (1,0))\geq 0$.
Also observe that
\begin{align*}
\frac{s}{m^{s+1}}\geq \int_{m}^{m+1}\frac{s}{x^{s+1}}\mathrm{d}x
=
\frac{1}{m^s}-\frac{1}{(m+1)^s}\geq 0.
\end{align*}
Thus, we have that
\begin{align*}
\sum_{(m,n)\in A}\frac{1}{(mn)^s}-\sum_{(m,n)\in A}\frac{1}{((m+1)n)^s}
=&
\sum_{(m,n)\in A}\frac{1}{n^s}\left(\frac{1}{m^s}-\frac{1}{(m+1)^s}\right)
\\
\leq&
\sum_{(m,n)\in A}\frac{1}{n^s}\frac{s}{m^{s+1}}
\\
\leq&
\sum_{n\in \positiveset}\frac{1}{n^s}\sum_{m\in \suph(A)}\frac{s}{m^{s+1}}.
\end{align*}
Dividing both sides by $\zeta^2(s)$ and letting $s\downarrow 1$,
since the last series is convergent as $s\downarrow 1$,
yields
\begin{align*}
\dens(A)-\dens(A\oplus (1,0))\leq 0.
\end{align*}

\end{proof}

\begin{proposition}[Translation Invariance]
\label{proposition.translatioin.invariance}
Let $A$ be a set of Gaussian integers.
Then
\begin{align*}
\dens(A \oplus (m,n)) = \dens(A),
\end{align*}
where $m$ and $n$ are nonnegative integers.
\end{proposition}

\begin{proof}
We have already proven that $\dens(A \oplus (1,0))=\dens(A \oplus (0,1))=\dens(A)$.
Therefore, we have that
\begin{align*}
\dens(A \oplus (m+1,n+1))
& = \dens(((A \oplus (m,n)) \oplus (1,0)) \oplus (0,1)) \\
&= \dens((A \oplus (m,n)) \oplus (1,0))\\
&= \dens(A \oplus (m,n))\\
&=\dens(A).
\end{align*}
\end{proof}

\begin{corollary}
The proposed density is not $\sigma$-additive.
\end{corollary}
\begin{proof}
This result follows directly from Propositions~\ref{proposition.facts}
and~\ref{proposition.translatioin.invariance}.
\end{proof}

Now consider the set operation defined as
$(a,b) \otimes A \triangleq \{ (am,bn) \ |\ (m,n)\in A\}$, where $(a,b)$ is a Gaussian integer.
This construction can be interpreted as
a dilation on the elements of $A$.
The following proposition relates the density of a set of Gaussian integers
with the density of its dilated form.

\begin{proposition}[Dilation]
\label{proposition.scaling}
Let $A$ be a set of Gaussian integers and let $(a,b)$ be any Gaussian integer.
Then
\begin{align*}
\dens((a,b) \otimes A) = \frac{1}{ab}\dens(A).
\end{align*}
\end{proposition}

\begin{proof}
This result follows directly from the definition of the proposed density:
\begin{align*}
\dens((a,b) \otimes A)
&=
\lim_{s\downarrow1}
\frac{\sum_{(m,n)\in A} \frac{1}{(ambn)^s}}
{\zeta^2(s)}
\\
&=
\lim_{s\downarrow1}
\frac{\frac{1}{(ab)^s}\sum_{(m,n)\in A} \frac{1}{(mn)^s}}
{\zeta^2(s)}
\\
&=
\frac{1}{ab}
\dens(A).
\end{align*}
\end{proof}

\section{Density of Particular Sets}

In this section,
we evaluate the density of some particular
sets of Gaussian integers.

\subsection{Cartesian Product of Arithmetic Progressions}

Let $p$ be an integer.
The set $M_{p} = \{m \in \positiveset : m \equiv 0 \pmod{p} \}$ constitutes
an arithmetic progression
with Dirichlet density $\dirichlet(M_p) = 1/p$.
Furthermore,
the Cartesian product of two arithmetic progressions generates a rectangular lattice
denoted by
$M_{(p,q)}\triangleq M_p \times M_q$,
where $p$ and $q$ are positive integers.
Then
it follows from Proposition~\ref{proposition.cartesian.product}
that
$\dens(M_{(p,q)}) = \dirichlet(M_p)\dirichlet(M_q)$.
Let us investigate the density of sets that are
intersections of particular Cartesian products of arithmetic progressions.

\begin{proposition}[Intersection]
\label{proposition.intersection}
For any positive integers $p$, $q$, $s$ and $t$,
we have that
\begin{align*}
\dens(M_{(p,q)} \cap M_{(s,t)})
=
\dens\left(M_{(\lcm(p,s),\lcm(q,t))}\right)
=
\frac{1}{\lcm(p,s)\lcm(q,t)},
\end{align*}
where $\lcm(\cdot,\cdot)$ denotes the least common multiple of its arguments.
\end{proposition}

\begin{proof}
First, note that
$
M_{(p,q)}=(p,q)\otimes \positiveset^2
$.
Therefore,
\begin{align*}
M_{(p,q)}\cap M_{(s,t)}
&=((p,q)\otimes \positiveset^2)\cap ((s,t)\otimes \positiveset^2)\nonumber\\
&=(\lcm(p,s),\lcm(q,t))\otimes \positiveset^2.
\end{align*}
Applying $\dens(\cdot)$ on both sides of above equation
and
invoking the dilation property,
we obtain the desired result.

\end{proof}

\begin{corollary}
\label{proposition.m.independence}
Let $(m,n)$ be a Gaussian integer.
Admit also that $\gcd(p,s)=1$ and $\gcd(q,t)=1$,
where $\gcd(\cdot,\cdot)$ returns the greatest common divisor of its arguments.
Then
\begin{align*}
\dens(M_{(mp,nq)} \cap M_{(ms,nt)}) = \frac{1}{mn} \dens(M_{(p,q)} \cap M_{(s,t)}).
\end{align*}
\end{corollary}

\begin{proof}
Follows directly from Proposition~\ref{proposition.intersection}.
\end{proof}

\subsection{Sets Delimited by Functions}

Let us consider a set of Gaussian integers
defined as
$
C = \{ (m,n) \in \positiveset^2 : f(m) \leq n \leq g(m) \},
$
where $f(\cdot)$ and $g(\cdot)$ are
functions
such that
$g(m)\geq f(m)\geq 1$
for every integer $m$.
Functions $f$ and $g$ delimit the set $C$,
confining the set elements in between.
Figure~\ref{fig.cone}
illustrates a possible configuration for the set $C$.
By definition,
the proposed density of $C$ is given by
\begin{align*}
\dens(C)=
\lim_{s\downarrow1}
\frac{1}{\zeta^2(s)}
\sum_{m=1}^\infty
\frac{1}{m^s}
\sum_{n = \lceil  f(m) \rceil}^{\lfloor  g(m) \rfloor}
\frac{1}{n^s},
\end{align*}
where $\lceil \cdot \rceil$ and $\lfloor \cdot \rfloor$
represent the usual ceiling and floor functions, respectively.

\begin{figure}[h]
\centering
\input{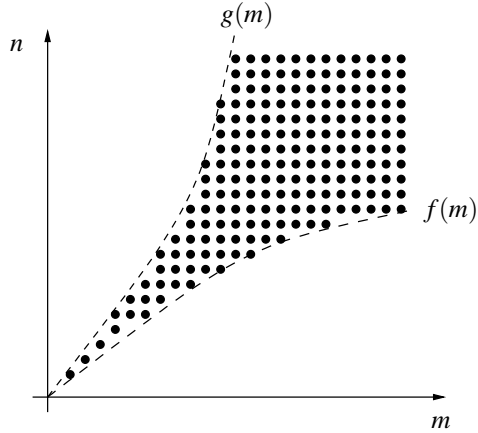}
\caption{A set upper and lower bounded by two functions.}
\label{fig.cone}
\end{figure}

Let us establish upper and lower bounds for the double summation.
Initially,
notice that the inner summation satisfies the following
bounds:
\begin{align*}
\sum_{n = \lceil  f(m) \rceil}^{\lfloor  g(m) \rfloor}
\frac{1}{n^s}
&=
\frac{1}{\lceil  f(m) \rceil^s}
+
\sum_{n = \lceil  f(m) \rceil + 1}^{\lfloor  g(m) \rfloor}
\frac{1}{n^s}
\\
&\leq
1
+
\int_{\lceil  f(m) \rceil}^{\lfloor  g(m) \rfloor}
\frac{1}{x^s}
\mathrm{d}{x}
\\
&\leq
1
+
\int_{ f(m)}^{ g(m)}
\frac{1}{x^s}
\mathrm{d}{x}
\\
&=
1
+
\frac{1}{-s+1}
\left(  g(m)^{-s+1} -  f(m)^{-s+1}
\right).
\end{align*}
Thus,
an upper bound for the double summation
is expressed by
\begin{align*}
\sum_{m=1}^\infty
\frac{1}{m^s}
\sum_{n = \lceil  f(m) \rceil}^{\lfloor  g(m) \rfloor}
\frac{1}{n^s}
&\leq
\sum_{m=1}^\infty
\frac{1}{m^s}
\left(
1
+
\frac{1}{-s+1}
\left(  g(m)^{-s+1} -  f(m)^{-s+1}
\right)
\right)
\\
&=
\zeta(s)
+
\frac{1}{-s+1}
\sum_{m=1}^\infty
\frac{1}{m^s}
\left(
g(m)^{-s+1} -  f(m)^{-s+1}
\right).
\end{align*}
Performing analogous manipulations,
we obtain the following lower bound for the inner summation:
\begin{align*}
\sum_{n = \lceil  f(m) \rceil}^{\lfloor  g(m) \rfloor}
\frac{1}{n^s}
&\geq
-\frac{1}{(\lceil  f(m) \rceil -1)^s}
+
\int_{\lceil  f(m)\rceil-1}^{\lfloor  g(m)\rfloor +1} \frac{1}{x^s} \mathrm{d}x
\\
&\geq
-1
+
\int_{ f(m)}^{ g(m)} \frac{1}{x^s} \mathrm{d}x
\\
&=
-1
+
\frac{1}{-s+1}
\left(
g(m)^{-s+1} - f(m)^{-s+1}
\right).
\end{align*}
This implies that the double summation is
lower bounded by:
\begin{align*}
\sum_{m=1}^\infty
\frac{1}{m^s}
\sum_{n = \lceil  f(m) \rceil}^{\lfloor  g(m) \rfloor}
\frac{1}{n^s}
&\geq
\sum_{m=1}^\infty
\frac{1}{m^s}
\left(
-1
+
\frac{1}{-s+1}
\left(
g(m)^{-s+1} - f(m)^{-s+1}
\right)
\right)
\\
&=
-\zeta(s)
+
\frac{1}{-s+1}
\sum_{m=1}^\infty
\frac{1}{m^s}
\left(
g(m)^{-s+1} -  f(m)^{-s+1}
\right).
\end{align*}

The upper and lower bounds present similar
formulations,
inviting an application of the squeeze theorem.
Thus,
after dividing both
expressions by $\zeta^2(s)$ and taking
the limit as $s\downarrow1$,
minor manipulations furnish
\begin{align*}
\dens(C)
=
\lim_{s\downarrow1}
\frac{1}{\zeta(s)}
\sum_{m=1}^\infty
\frac{1}{m^s}
\left(
f(m)^{-s+1} - g(m)^{-s+1}
\right).
\end{align*}

Now let us analyze the density of a set $C$ in the light of the asymptotic behavior of
the delimiting functions.

\begin{proposition}
[Asymptotics]
Let
$u(m)$ and $v(m)$ be delimiting functions that are always greater or equal to one.
If $f(m) = \Theta(u(m))$ and $g(m) = \Theta(v(m))$,
then
\begin{align}
\label{density.cone}
\dens(C)
=
\lim_{s\downarrow1}
\frac{1}{\zeta(s)}
\sum_{m=1}^\infty
\frac{1}{m^s}
\left(
u(m)^{-s+1} - v(m)^{-s+1}
\right).
\end{align}
\end{proposition}

\begin{proof}
By the definition of the $\Theta$-notation~\cite[p.~434]{graham1989},
there exist a quantity $m_0$,
such that,
for every $m\geq m_0$,
both functions $f$ and $g$ satisfy:
\begin{align*}
c_1 u(m) \leq f(m) \leq c_2 u(m), \\
c_3 v(m) \leq g(m) \leq c_4 v(m),
\end{align*}
where $c_1$, $c_2$, $c_3$, and $c_4$ are positive constants.
Moreover,
notice that
\begin{align*}
\dens(C)
=&
\lim_{s\downarrow1}
\frac{1}{\zeta(s)}
\sum_{m=1}^\infty
\frac{1}{m^s}
\left(
f(m)^{-s+1} - g(m)^{-s+1}
\right)
\\
=&
\lim_{s\downarrow1}
\frac{1}{\zeta(s)}
\left[
\sum_{m=1}^{m_0-1}
\frac{1}{m^s}
\left(
f(m)^{-s+1} - g(m)^{-s+1}
\right)
\right]
\\
&+
\lim_{s\downarrow1}
\frac{1}{\zeta(s)}
\left[
\sum_{m=m_0}^\infty
\frac{1}{m^s}
\left(
f(m)^{-s+1} - g(m)^{-s+1}
\right)
\right]
\\
=&
\lim_{s\downarrow1}
\frac{1}{\zeta(s)}
\sum_{m=m_0}^\infty
\frac{1}{m^s}
\left(
f(m)^{-s+1} - g(m)^{-s+1}
\right).
\end{align*}
Thus, for $m\geq m_0$,
we have that
\begin{align*}
\sum_{m=m_0}^\infty
&\frac{1}{m^s}
\left(
(c_2 u(m))^{-s+1} - (c_3 v(m))^{-s+1}
\right)
\\
&\leq
\sum_{m=m_0}^\infty
\frac{1}{m^s}
\left(
f(m)^{-s+1} - g(m)^{-s+1}
\right)
\\
&\leq
\sum_{m=m_0}^\infty
\frac{1}{m^s}
\left(
(c_1 u(m))^{-s+1} - (c_4 v(m))^{-s+1}
\right).
\end{align*}
Now we show that after dividing by $\zeta(s)$ and letting $s\downarrow1$,
both upper and lower bounds above have the same limit.
Since $u(m)\geq 1$ and $v(m)\geq 1$,
it follows that for arbitrary positive constants $k_1$ and $k_2$:
\begin{align*}
\sum_{m=m_0}^{\infty}
\frac{1}{m^s}
&
((k_1u(m))^{-s+1}-(k_2v(m))^{-s+1})
=
\\
&k_1^{-s+1}
\sum_{m=m_0}^{\infty}
\frac{1}{m^s}
u(m)^{-s+1}
-
k_2^{-s+1}
\sum_{m=m_0}^{\infty}
\frac{1}{m^s}
v(m)^{-s+1}.
\end{align*}
Thus,
since both $k_1^{-s+1}$ and $k_2^{-s+1}$ tend to one as $s\downarrow 1$,
we have that
\begin{align*}
\lim_{s\downarrow 1}\frac{1}{\zeta(s)}
\sum_{m=m_0}^{\infty}\frac{1}{m^s}((k_1u(m))^{-s+1}-(k_2v(m))^{-s+1})\nonumber\\
=
\lim_{s\downarrow 1}\frac{1}{\zeta(s)}
\sum_{m=m_0}^{\infty}\frac{1}{m^s}(u(m)^{-s+1}-v(m)^{-s+1}).
\end{align*}
Therefore, we maintain that
\begin{align*}
&\lim_{s\downarrow1}
\frac{1}{\zeta(s)}
\sum_{m=m_0}^\infty
\frac{1}{m^s}
\left(
f(m)^{-s+1} - g(m)^{-s+1}
\right)
\\
&=
\lim_{s\downarrow1}
\frac{1}{\zeta(s)}
\sum_{m=m_0}^\infty
\frac{1}{m^s}
\left(
u(m)^{-s+1} - v(m)^{-s+1}
\right).
\end{align*}
Finally,
since
\begin{align*}
\lim_{s\downarrow1}
\frac{1}{\zeta(s)}
\sum_{m=1}^{m_0-1}
\frac{1}{m^s}
\left(
u(m)^{-s+1} - v(m)^{-s+1}
\right)
=
0,
\end{align*}
the proposition is proven.
\end{proof}

We now supply two examples.
But,
the following lemma is needed before.
\begin{lemma}
For $\alpha\geq0$,
$
\lim_{s\downarrow1} \zeta((\alpha+1)s - \alpha)(s-1) = (1+\alpha)^{-1}.
$

\end{lemma}
\begin{proof}
Taking into account the substitution $t = (\alpha+1)s-\alpha$,
it follows that:
\begin{align*}
\lim_{s\downarrow1} \zeta((\alpha+1)s - \alpha)(s-1)
&=
\lim_{t\downarrow1} \zeta(t)\left(\frac{t+\alpha}{1+\alpha}-1\right)
\\
&=\lim_{t\downarrow1} \zeta(t)\frac{t-1}{1+\alpha}
\\
&=\frac{1}{1+\alpha}.
\end{align*}
\end{proof}

\begin{example}
Let us examine the density of the set
$C_\text{pow} = \{ (m,n) \in \positiveset^2 : f(m) \leq n \leq g(m) \}$,
where $g(m) = \Theta(m^\beta)$ and $f(m) = \Theta(m^\alpha)$,
for real quantities
$\beta\geq\alpha>0$.
In order to compute such density we need the previous lemma.
\noindent
Thus,
by invoking Equation~\ref{density.cone},
it follows that the sought density is given by
\begin{align*}
\dens(C_\text{pow})
&=
\lim_{s\downarrow1}
\frac{1}{\zeta(s)}
\sum_{m=1}^\infty
\frac{1}{m^s}
\left(
(m^\alpha)^{(-s+1)}
-
(m^\beta)^{(-s+1)}
\right)
\\
&=
\lim_{s\downarrow1}
\frac{1}{\zeta(s)}
(
\zeta((\alpha+1)s-\alpha)
-
\zeta((\beta+1)s-\beta)
)
\\
&=
\frac{1}{1+\alpha}
-
\frac{1}{1+\beta}.
\end{align*}

In particular, if $\alpha \beta  = 1$,
we have
$\dens(C_\text{pow}) = \frac{\beta - 1}{\beta + 1}$.
\end{example}

\begin{example}
Consider the set
$C_\text{exp}=
\{(m,n)\in\positiveset^2 :
n \leq g(m) \}$,
where $g(m) = \Theta(a^m)$,
for $a>1$.
Thus,
by Equation~\ref{density.cone},
\begin{align*}
\dens(C_\text{exp})= \lim_{s\downarrow1}\frac{1}{\zeta(s)}\left(\zeta(s) - \sum_{m=1}^{\infty}\frac{(a^m)^{-s+1}}{m^s}\right).
\end{align*}
Then, note that for each $\beta>0$,
there is a quantity $M$ such that $m\geq M$ implies that $a^m \geq m^\beta$.
By Example 1, we know that
\begin{align*}
\lim_{s\downarrow 1}\frac{1}{\zeta(s)}\sum_{m=1}^{\infty}\frac{(m^{\beta})^{-s+1}}{m^s}=\frac{1}{1+\beta}.
\end{align*}
Moreover, since $\lim_{s\downarrow 1}\frac{1}{\zeta(s)}\sum_{m=1}^{M-1}\frac{(m^{\beta})^{-s+1}}{m^s}=0$, it follows that
\begin{align*}
\lim_{s\downarrow 1}\frac{1}{\zeta(s)}\sum_{m=M}^{\infty}\frac{(m^{\beta})^{-s+1}}{m^s}=\frac{1}{1+\beta}.
\end{align*}
Thus,
\begin{align*}
\dens(C_\text{exp})
&\geq
\lim_{s\downarrow1}
\frac{1}{\zeta(s)}
\left(
\zeta(s) - (\sum_{m=1}^{M-1}\frac{(a^m)^{-s+1}}{m^s}+\sum_{m=M}^{\infty}\frac{(m^{\beta})^{-s+1}}{m^s})
\right)
\\
&=
1-
\left(0+\frac{1}{1+\beta}\right)
=
\frac{\beta}{\beta+1}.
\end{align*}
Finally, letting $\beta\to\infty$
yields
$\dens(C_\text{exp}) = 1$.
\end{example}

\section*{Acknowledgements}

The
second
author acknowledges partial financial support from
the
Department of Foreign Affairs and International Trade (DFAIT), Canada,
and
the
\emph{Conselho Nacional de Desenvolvimento Cient\'ifico e Tecnol\'ogico}
(CNPq), Brazil.

{\small
\bibliographystyle{IEEEtran}
\bibliography{ref}

\begin{thebibliography}{10}
\providecommand{\url}[1]{#1}
\csname url@samestyle\endcsname
\providecommand{\newblock}{\relax}
\providecommand{\bibinfo}[2]{#2}
\providecommand{\BIBentrySTDinterwordspacing}{\spaceskip=0pt\relax}
\providecommand{\BIBentryALTinterwordstretchfactor}{4}
\providecommand{\BIBentryALTinterwordspacing}{\spaceskip=\fontdimen2\font plus
\BIBentryALTinterwordstretchfactor\fontdimen3\font minus
  \fontdimen4\font\relax}
\providecommand{\BIBforeignlanguage}[2]{{%
\expandafter\ifx\csname l@#1\endcsname\relax
\typeout{** WARNING: IEEEtran.bst: No hyphenation pattern has been}%
\typeout{** loaded for the language `#1'. Using the pattern for}%
\typeout{** the default language instead.}%
\else
\language=\csname l@#1\endcsname
\fi
#2}}
\providecommand{\BIBdecl}{\relax}
\BIBdecl

\bibitem{sun2008density}
X.-G. Sun and J.-H. Fand, ``On the density of integers of the form
  $(p-1)2^{-n}$ in arithmetic progression,'' \emph{Bulletin of the Australian
  Mathematical Society}, vol.~78, no.~3, pp. 431--436, Dec. 2008.

\bibitem{bell2006}
J.~P. Bell and S.~N. Burris, ``{Dirichlet} density extends global asymptotic
  density in multiplicative systems,'' Preprint published at
  {htt://www.math.uwaterloo.ca/$\sim$snburris/htdocs/MYWORKS/}, 2006.

\bibitem{duncan1968}
R.~L. Duncan, ``Some continuity properties of the {Schnirelmann} density,''
  \emph{Pacific Journal of Mathematics}, vol.~26, no.~1, pp. 57--58, 1968.

\bibitem{duncan1970}
------, ``Some continuity properties of the {Schnirelmann} density {II},''
  \emph{Pacific Journal of Mathematics}, vol.~32, no.~1, pp. 65--67, 1970.

\bibitem{erdos1948}
P.~Erd{\"o}s, ``On the density of some sequences of integers,'' \emph{Bulletin
  of the American Mathematical Society}, vol.~54, no.~8, pp. 685--692, Aug.
  1948.

\bibitem{davenport1951}
H.~Davenport and P.~Erd{\"o}s, ``On sequences of positive integers,''
  \emph{Journal of the Indian Mathematical Society}, vol.~15, pp. 19--24, 1951.

\bibitem{ahlswede1997}
R.~Ahlswede and L.~H. Khachatrian, ``Number theoretic correlation inequalities
  for {Dirichlet} densities,'' \emph{Journal of Number Theory}, vol.~63, pp.
  34--46, 1997.

\bibitem{niven1951asymptotic}
I.~Niven, ``The asymptotic density of sequences,'' \emph{Bulletin of the
  American Mathematical Society}, vol.~57, no.~6, pp. 420--434, 1951.

\bibitem{grad1965}
I.~S. Gradshteyn and I.~M. Ryzhik, \emph{Table of Integrals, Series, and
  Products}, 4th~ed.\hskip 1em plus 0.5em minus 0.4em\relax New York: Academic
  Press, 1965.

\bibitem{bateman2004}
P.~T. Bateman and H.~G. Diamond, \emph{Analytic Number Theory}.\hskip 1em plus
  0.5em minus 0.4em\relax World Scientific, 2004.

\bibitem{ahlswede1996}
R.~Ahlswede and L.~H. Khachatrian, \emph{The Mathematics of {P}.
  {E}rd\"os}.\hskip 1em plus 0.5em minus 0.4em\relax Berlin: Springer Verlag,
  1996, vol. Vol. I, Algorithms and Combinatorics B, ch. Classical results on
  primitive and recent results on cross-primitive sequences, pp. 104--116.

\bibitem{ahlswede1999}
R.~Ahlswede, L.~H. Khachatrian, and A.~S\'ark\"ozy, ``On the counting function
  for primitive sets of integers,'' \emph{Journal of Number Theory}, vol.~79,
  pp. 330--344, 1999.

\bibitem{ahlswede2004}
------, ``On the density of primitive sets,'' \emph{Journal of Number Theory},
  vol. 109, pp. 319--361, 2004.

\bibitem{fuchs1990}
A.~Fuchs and R.~G. Antonini, ``Th\'eorie g\'en\'erale des densit\'es,''
  \emph{Memorie di Matematica}, vol. XIV, pp. 253--294, 1990.

\bibitem{cheo1951}
L.~Cheo, ``On the density of sets of {G}aussian integers,'' \emph{The American
  Mathematical Monthly}, vol.~58, pp. 618--620, 1951.

\bibitem{kvarda1963}
B.~Kvarda, ``On densities of sets of lattice points,'' \emph{Pacific Journal of
  Mathematics}, vol.~13, no.~2, pp. 611--615, 1963.

\bibitem{kvarda1966}
B.~Kvarda and R.~Killgrove, ``Extensions of the {S}chnirelmann density to
  higher dimensions,'' \emph{The American Mathematical Monthly}, vol.~73,
  no.~9, pp. 976--979, Nov. 1966.

\bibitem{stalley1955}
R.~Stalley, ``A modified {S}chnirelmann density,'' \emph{Pacific Journal of
  Mathematics}, vol.~5, no.~1, pp. 119--124, 1955.

\bibitem{freedman1970}
A.~R. Freedman, ``On the generalization of the modified {S}chnirelmann
  density,'' \emph{Journal of Number Theory}, vol.~2, no.~1, pp. 97--105, 1970.

\bibitem{freedman1969}
------, ``On asymptotic density in $n$-dimensions,'' \emph{Pacific Journal of
  Mathematics}, vol.~29, no.~1, pp. 95--113, 1969.

\bibitem{freedman1973}
------, ``On the additivity theorem for $n$-dimensional asymptotic density,''
  \emph{Pacific Journal of Mathematics}, vol.~49, no.~2, pp. 357--363, 1973.

\bibitem{rudin1966}
W.~Rudin, \emph{Real and Complex Analysis}.\hskip 1em plus 0.5em minus
  0.4em\relax McGraw-Hill, 1966.

\bibitem{graham1989}
R.~L. Graham, D.~E. Knuth, and O.~Patashnik, \emph{Concrete Mathematics}.\hskip
  1em plus 0.5em minus 0.4em\relax Addison-Wesley Publishing Company, 1989.

\end{thebibliography}
}

\end{document}